\documentclass{amsart}

\usepackage{amsmath, amssymb, amsthm, mathrsfs, graphicx}
\usepackage{hyperref, enumitem, xspace, ifthen,comment}
\usepackage[all]{xy}
\usepackage{tikz}
\usetikzlibrary{arrows,shapes,trees}

\newtheorem*{thm-plain}{Theorem}
\newtheorem{thm}{Theorem}[section]
\newtheorem{lem}[thm]{Lemma}
\newtheorem{prp}[thm]{Proposition}
\newtheorem{cor}[thm]{Corollary}
\newtheorem{fct}[thm]{Fact}
\newtheorem{conj}[thm]{Conjecture}
\newtheorem{ques}[thm]{Question}
\numberwithin{equation}{thm}

\theoremstyle{definition}

\newtheorem*{dfn-plain}{Definition}

\theoremstyle{remark}
\newtheorem{clm}[thm]{Claim}

\newtheorem{ntn}[thm]{Notation}

\newtheorem*{rem-plain}{Remark}

\DeclareMathOperator{\sing}{sing}

\DeclareMathOperator{\reg}{reg}

\DeclareMathOperator{\Pic}{Pic}
\DeclareMathOperator{\codim}{codim}

\DeclareMathOperator{\Exc}{Exc}

\def\rd#1.{\lfloor{#1}\rfloor}
\def\rp#1.{\lceil{#1}\rceil}

\newcommand{\Q}{\ensuremath{\mathbb Q}}
\newcommand{\R}{\mathbb R}
\newcommand{\C}{\mathbb C}

\newcommand{\A}{\mathbb A}

\renewcommand{\O}{\mathscr O}

\newcommand{\x}{\times}

\renewcommand{\phi}{\varphi}
\renewcommand{\theta}{\vartheta}
\newcommand{\id}{\mathrm{id}}

\newcommand{\minus}{\setminus}
\newcommand{\inj}{\hookrightarrow}

\newcommand{\ms}{\mathscr}
\newcommand{\tensor}{\otimes}

\newcommand{\sHom}{\mathscr H\!om}

\newcommand{\red}{\mathrm{red}}

\newcommand{\nklt}{\mathrm{nklt}}

\newcommand{\sL}{\mathscr{L}}

\newcommand{\sO}{\mathscr{O}}

\newcommand{\sQ}{\mathscr{Q}}
\newcommand{\sR}{\mathscr{R}}

\newcommand{\sT}{\mathscr{T}}

\newdir{ ir}{{}*!/-5pt/@^{(}} 
\newdir{ il}{{}*!/-5pt/@_{(}} 

\DeclareMathOperator{\supp}{supp}

\newcommand{\eps}{\varepsilon}

\newcommand{\dif}{\mathrm d}

\newcommand{\Res}{\mathrm{Res}}             
\newcommand{\wt}{\widetilde}
\newcommand{\wh}{\widehat}

\newcommand{\bb}{B}

\newenvironment{sequation}{%
\setcounter{equation}{\value{thm}}%
\numberwithin{equation}{section}%
\begin{equation}%
}{%
\end{equation}%
\numberwithin{equation}{thm}%
\addtocounter{thm}{1}%
}

\numberwithin{equation}{thm}

\DeclareRobustCommand{\SkipTocEntry}[5]{}

\setitemize[1]{leftmargin=*,parsep=0em,itemsep=0.125em,topsep=0.125em}
\setenumerate[1]{leftmargin=*,parsep=0em,itemsep=0.125em,topsep=0.125em}

\newcommand{\iref}[3]{\the\value{#1}.\the\value{#2}(\the\value{#3})}

\newcommand\factor[2]{\left. \raise 2pt\hbox{$#1$} \right/\hskip -2pt \raise -2pt\hbox{$#2$}}

\definecolor{forrest}{RGB}{81,133,49}
\definecolor{mydarkblue}{RGB}{10,92,153}

\begin{document}

\title[An optimal extension theorem for $1$-forms]{An optimal extension theorem for $1$-forms \\ and the Lipman-Zariski conjecture}
\author{Patrick Graf} %
\author{S\'andor J Kov\'acs} %
\address{PG: Lehrstuhl f\"ur Mathematik I, Universit\"at Bayreuth,
  95440 Bayreuth, Germany} %
\email{\href{mailto:patrick.graf@uni-bayreuth.de}{patrick.graf@uni-bayreuth.de}}
\address{ SJK: University of Washington, Department of Mathematics, Box 354350,
  Seattle, WA 98195-4350, USA} %
\email{\href{mailto:skovacs@uw.edu}{skovacs@uw.edu}} \date{\today} %
\thanks{The first named author was partially supported by the
  DFG-Forschergruppe 790 ``Classification of Algebraic Surfaces and Compact Complex
  Manifolds''.}  %
\thanks{The second named author was supported in part by NSF Grants DMS-0856185,
  DMS-1301888 and the Craig McKibben and Sarah Merner Endowed Professorship in
  Mathematics at the University of Washington.} %

\begin{abstract}
  Let $X$ be a normal variety. Assume that for some reduced divisor $D \subset X$,
  logarithmic $1$-forms
  defined on the snc locus of $(X, D)$ extend to a log resolution $\wt X \to X$
  as logarithmic differential forms.
  We prove that then the Lipman-Zariski conjecture holds for $X$.
  This result applies in particular if $X$ has log canonical singularities.

  Furthermore, we give an example of a $2$-form defined on the smooth locus of a three-dimensional log canonical pair $(X, \emptyset)$ which acquires a logarithmic pole along an exceptional divisor of discrepancy zero, thereby improving on a similar example of Greb, Kebekus, Kov\'acs and Peternell.
\end{abstract}

\maketitle

\tableofcontents

\section{Introduction}

\subsection{Main results}

The Lipman-Zariski conjecture \cite{Lip65} asserts the following.

\begin{conj}\label{conj:LZ} 
  Let $X$ be a complex variety such that the tangent sheaf $\sT_X := \sHom_{\O_X}(\Omega_X^1, \O_X)$
  is locally free. Then $X$ is smooth.
\end{conj}

Despite being almost 50 years old, this conjecture remains open in general. It is
therefore natural to consider special cases. The conjecture is known to hold in the
case that the singular locus of $X$ has codimension at least 3 by the work of Flenner
\cite{Fle88} and for complete intersections by the work of K\"allstr\"om
\cite{Kal11}.
Recently, J\"order \cite{Jor13} proved the conjecture under the assumption that $\sT_X$ locally has a basis consisting of \emph{commuting} vector fields.
Earlier work on Conjecture~\ref{conj:LZ} includes \cite{SS72},
\cite{Hoc75}, and \cite{SvS85}.

In a slightly different direction, one may consider varieties with only singularities
arising in the minimal model program. The minimal model program aims at the
birational classification of varieties, and it is well-known that even if one is
interested only in smooth varieties, running the mmp requires dealing with singular
models.  For us the most important classes will be \emph{klt} (Kawamata log terminal)
and \emph{log canonical}. Klt singularities form the largest class of singularities
where most of the mmp is known to work, while the class of log
canonical singularities is the largest class where the relevant notions of the mmp
make sense.  It is also a more stable class than that of klt singularities; log
canonical singularities and their non-normal versions play an important role in
compactifications of moduli spaces of canonically polarized varieties.  For the
precise definitions, we refer to \cite[Sec.~2.3]{KM98}.
Greb-Kebekus-Kov\'acs-Peternell \cite[Theorem 6.1]{GKKP11} showed Conjecture
\ref{conj:LZ} for klt
spaces.

In this paper, we prove the following new special case of Conjecture~\ref{conj:LZ}.

\begin{thm}[Lipman-Zariski conjecture given an extension theorem] \label{thm:LZ ext}
  Let $X$ be a normal complex variety, and let $\pi\!: \wt X \to X$ be a log
  resolution. Assume that the sheaf
  \[ \pi_* \Omega_{\wt X}^1 (\log \wt D) \] is reflexive for some snc divisor $\wt D
  \subset \wt X$.  Then the Lipman-Zariski conjecture holds for $X$, i.e., $\sT_X$
  being locally free implies that $X$ is smooth.
\end{thm}

Note that Lipman \cite[Thm.~3]{Lip65} proved that if $\sT_X$ is locally free, then
$X$ is normal. Hence the normality assumption in our theorem is not a real
restriction.

Also note that the reflexivity assumption in the theorem is equivalent to saying that logarithmic $1$-forms defined on the snc locus of $(X, \pi_* \wt D)$ extend to logarithmic $1$-forms on $(\wt X, \wt D)$ -- cf.~\cite[Rem.~1.5.2]{GKKP11}.

By \cite[Thm.~1.5]{GKKP11}, we have the following immediate corollary.

\begin{cor}[Lipman-Zariski conjecture for log canonical pairs] \label{cor:LZ lc}
  Let $(X, D)$ be a complex log canonical pair. Then the Lipman-Zariski conjecture holds for $X$.
\end{cor}

By the same method of proof as in Theorem~\ref{thm:LZ ext}, we obtain the following result.

\begin{thm}[Extension theorem for 1-forms on log canonical pairs] \label{thm:Ext of 1-forms}
  Let $(X, D)$ be a complex log canonical pair, and let $\pi\!: \wt X \to X$ be a log
  resolution of $(X, D)$.  Then the sheaf
  \[ 
  \pi_* \Omega_{\wt X}^1(\log \wt D) 
  \] 
  is reflexive, where $\wt D$ is any reduced divisor such that
  \[ \Exc(\pi) \wedge \pi^{-1}(\rd D.)
  \subseteq \supp \wt D \subseteq
  \pi^{-1}(\rd D.). \]
\end{thm}

Here $\rd D.$ denotes the coefficient-wise round-down of $D$. In our case, $\rd D.$ is simply the union of all components of $D$ that have coefficient one.
The expression $\Exc(\pi) \wedge \pi^{-1}(\rd D.)$ denotes the largest divisor contained in both $\Exc(\pi)$ and $\pi^{-1}(\rd D.)$.

In a similar fashion as above, the reflexivity assertion is equivalent to saying that any logarithmic $1$-form defined on the snc locus of $(X, \pi_* \wt D)$
can be extended to $\wt X$, possibly acquiring logarithmic poles along $\wt D$.

Theorem~\ref{thm:Ext of 1-forms} should be compared to the extension theorem \cite[Thm.~1.5]{GKKP11}.
There the conclusion is similar, with $\wt D$ replaced by $\wh D$, the
largest reduced divisor contained in $\pi^{-1}(\text{non-klt locus})$. (The non-klt
locus is the smallest closed subset $W \subset X$ such that $(X, D)$ is klt away
from $W$. Note that this contains $\rd D.$.) Theorem~\ref{thm:Ext of 1-forms} says
that~\cite[Thm.~1.5]{GKKP11} is not optimal: we allow logarithmic poles only along a smaller divisor. For example, if $D = \emptyset$ but $X$ is not klt, e.g.~if $X$ is a cone over an abelian variety, then $\wt D = 0$ while $\wh D$ is nonzero.

\subsection{Further results}

We show that Theorem~\ref{thm:Ext of 1-forms} in turn is optimal, both with
respect to the pole divisor $\wt D$ and with respect to the degree of the forms
considered. To be more precise, concerning the first point we prove the following.

\begin{thm}[{Optimality of Theorem~\ref{thm:Ext of 1-forms}}] \label{thm:opt}
  Let $(X, D)$ and $\pi$ be as in Theorem~\ref{thm:Ext of 1-forms},
  and let $\wt D = \pi^{-1}(\rd D.)$.
  Assume that one of the following holds:
  \begin{enumerate}
  \item\label{itm:opt-Qfact}   $X$ is \Q-factorial, or
  \item\label{itm:opt-surface} $\dim X = 2$.
  \end{enumerate}
  Then for any divisor $\bb$ such that $\pi^{-1}_* \rd D. \leq B \lneq \wt D$, the sheaf
  $\pi_* \Omega_{\wt X}^1(\log \bb)$ is \emph{not} reflexive.  
\end{thm}

As to the second point, we show that for resolutions of log canonical pairs, forms of higher degree may acquire logarithmic poles along exceptional divisors of discrepancy strictly greater than $-1$, even if the boundary divisor of the pair is empty. This improves upon an example given in \cite[Ex.~3.2]{GKKP11}. The precise statement is as follows.

\begin{thm}[Non-extension over klt places, cf.~Theorem~\ref{thm:klt non-ext 2}]
  \label{thm:klt non-ext} There exists a three-dimensional complex log canonical pair $(X,
  \emptyset)$ with empty boundary such that there is a divisor $E_0 \subset \wt X$ of
  discrepancy $0$ in some log resolution $\wt X \to X$, and a $2$-form on the smooth
  locus of $X$ that acquires a logarithmic pole along $E_0$ when pulled back to $\wt
  X$.
\end{thm}

\subsection{Overview of proofs}

The proofs of our main results are based on the two auxiliary Theorems~\ref{thm:drop-a-non-exceptional-div} and~\ref{thm:ext along exc}.
The purpose of these theorems is to shrink the divisor along which we allow pulled-back $1$-forms to acquire logarithmic poles.
Theorem~\ref{thm:drop-a-non-exceptional-div} deals with non-exceptional components, while Theorem~\ref{thm:ext along exc} handles the exceptional ones.

To prove Theorem~\ref{thm:LZ ext}, we first apply Theorem~\ref{thm:drop-a-non-exceptional-div} and then Theorem~\ref{thm:ext along exc} in order to shrink that pole divisor to zero. Then an argument going back to \cite[(1.6)]{SvS85} completes the proof.
To prove Theorem~\ref{thm:Ext of 1-forms}, we take~\cite[Thm.~16.1]{GKKP11} as our starting point and then apply Theorems~\ref{thm:drop-a-non-exceptional-div} and~\ref{thm:ext along exc}.

\subsection{Recent work by Druel}

In \cite[Thm.~1.1]{Dru13}, Druel has recently obtained Corollary~\ref{cor:LZ lc} by
an independent proof.
He employs a cutting-down procedure to reduce to the surface case, where the main work is done.  Note however
that this case is essentially already contained in \cite{SvS85}. To be more precise,
if $x \in X$ is a normal surface singularity with smooth locus $U$ and $\pi\!: \wt X
\to X$ is a log resolution with exceptional divisor $E$, then \cite[Cor.~1.4]{SvS85}
says that the map
\[ 
\factor{ H^0(U, \Omega_U^1) }{ H^0(\wt X, \Omega_{\wt X}^1) } \longrightarrow
\factor{ H^0(U, \Omega_U^2) }{ H^0(\wt X, \Omega_{\wt X}^2(\log E)) } 
\]
induced by differentiation is injective. If $x \in X$ is log canonical, then by
definition the right-hand side is zero, hence so is the left-hand side. This means
that all $1$-forms defined on $U$ extend to $\wt X$ without poles. Now the argument
given in \cite[(1.6)]{SvS85} shows that if $\sT_X$ is free, then $x \in X$ is in fact
smooth.

\subsection{Acknowledgements}
The authors would like to thank Daniel Greb, Clemens J\"order and Stefan Kebekus for
interesting discussions on the subject of this paper.

\subsection{Notation, definitions, and conventions}

Throughout this paper, we work over the field of complex numbers $\C$.

A \emph{pair} $(X, D)$ consists of a normal variety $X$ over $\C$ and an effective $\R$-Weil divisor $D$ on $X$.
  
Let $(X, D)$ be a pair and $x \in X$ a point. We say that $(X, D)$ is
\emph{snc at $x$} if there exists a Zariski-open neighbourhood $U \subseteq X$ of $x$ such that $U$ is smooth and $\supp D \cap U$ is either empty, or a divisor with simple normal crossings.  The pair $(X, D)$ is called an \emph{snc pair} or simply \emph{snc} if it is snc at every point of $X$.

For the definitions of \emph{klt} and \emph{log canonical} pairs, we refer to~\cite[Sec.~2.3]{KM98}.

Given a pair $(X, D)$, let $(X, D)_{\reg}$ denote the maximal open
subset of $X$ where $(X, D)$ is snc, and $(X, D)_{\sing}$ its complement,
with the induced reduced subscheme structure.

Let $(X, D)$ be a pair. A \emph{log resolution} of $(X, D)$ is
a proper birational morphism $\pi\!: \wt X \to X$ such that $\wt X$ is smooth,
both the pre-image $\pi^{-1}(\supp D)$ of $\supp D$ and the exceptional set
$E = \Exc(\pi)$ are of pure codimension one in $\wt X$, and $(\wt X, \wt D + E)$ is an snc pair where $\wt D = \pi^{-1}(\supp D)_{\red}$ is the reduced divisor supported on $\pi^{-1}(\supp D)$.

Let $D$ be a divisor on a normal variety, and let $D = \sum a_i D_i$ be its decomposition into irreducible components. The \emph{round-down} $\rd D.$ of $D$ is defined to be $\sum \rd a_i. D_i$, where $\rd a_i.$ is the largest integer less than or equal to $a_i$.

Let $D_1, D_2$ be divisors on a normal variety. Then $D_1 \vee D_2$ denotes the
smallest divisor that contains both $D_1$ and $D_2$, while $D_1 \wedge D_2$ denotes the largest divisor that is contained in both $D_1$ and $D_2$.

\section{Dropping non-exceptional divisors}

In this section we prove that if the extension theorem holds for a pair $(X, \Gamma + \Delta)$ where $\Delta$ is a reduced effective divisor, then it also holds for $(X, \Gamma)$. More precisely we prove the following.

\begin{thm}[Dropping non-exceptional divisors] \label{thm:drop-a-non-exceptional-div}
  Let $X$ be a normal variety and $\pi\!: \wt X \to X$ a log resolution of
  $X$. Assume that the sheaf
  \[ \pi_* \Omega_{\wt X}^1(\log \wt D) \]
  is reflexive for some snc divisor $\wt D$. Let $\Delta$ be a reduced effective
  divisor on $X$ such that $\supp \Delta \subset \pi_* \wt D$. Then the sheaf
  \[ \pi_* \Omega_{\wt X}^1(\log \wt B) \]
  is also reflexive, where $\wt B = \wt D - \pi^{-1}_* \Delta$.
\end{thm}

\begin{proof}
Notice that one may assume that $\Delta$ is irreducible and conclude the general case via replacing $\wt D$ by $\wt B$ and iterating the process for all irreducible components of $\Delta$. For simplicity let us denote $\pi^{-1}_* \Delta$ by $\wt \Delta$.
Consider the following short exact sequence given by the residue map, cf.~\cite[2.3(b)]{EV92}:
  \[
  0 \to \Omega_{\wt X}^1(\log \wt B) \to \Omega_{\wt X}^1(\log \wt D) \to
  \sO_{\wt\Delta} \to 0.
  \]
  Pushing this forward via $\pi$ gives
  \[
  0 \to \pi_* \Omega_{\wt X}^1(\log \wt B) \to \pi_*\Omega_{\wt X}^1(\log \wt D) \to
  \sQ \to 0,
  \]
  where $\sQ \subset \pi_* \sO_{\wt\Delta}$.
  In particular, $\sQ$ is supported on $\Delta$ and it is torsion-free as an $\sO_{\Delta}$-module. It follows that the only associated prime of $\sQ$ has height $1$. Then the statement follows from \cite[Cor.~1.5]{MR597077}.
\end{proof}

\section{Dropping certain exceptional divisors}

The aim of the present section is to show that if $\omega$ is a logarithmic $1$-form
on a smooth variety whose poles are contained in an exceptional divisor, then $\omega$ in
fact does not have any poles.

\begin{thm}[Dropping exceptional divisors] \label{thm:ext along exc}
Let $X$ be a normal variety and $\pi\!: \wt X \to X$ a log resolution. Let $E$ be a reduced $\pi$-exceptional divisor. Then the natural inclusion map
\[ H^0\big(\wt X, \Omega_{\wt X}^1\big) \inj H^0\big(\wt X, \Omega_{\wt X}^1(\log E)\big) \]
is an isomorphism.
Equivalently, the inclusion $\pi_* \Omega_{\wt X}^1 \subset \pi_* \Omega_{\wt X}^1(\log E)$ is an isomorphism of sheaves.
\end{thm}

Theorem~\ref{thm:ext along exc} is a consequence of the following two propositions.

\begin{prp}[{Theorem~\ref{thm:ext along exc} for isolated singularities}] \label{prp:ext along exc-pt}
Let $X$ be a normal variety and $\pi\!: \wt X \to X$ a log resolution. Let $E$ be a reduced divisor which is mapped to a single point by $\pi$. Then the natural inclusion map
\[ H^0\big(\wt X, \Omega_{\wt X}^1\big) \inj H^0\big(\wt X, \Omega_{\wt X}^1(\log E)\big) \]
is an isomorphism.
\end{prp}

\begin{prp} \label{prp:ext along exc-reduction}
  Proposition~{\rm \ref{prp:ext along exc-pt}} implies Theorem~{\rm \ref{thm:ext
      along exc}}.
\end{prp}

Proposition~\ref{prp:ext along exc-pt} was first observed by Wahl in the case of
surfaces, cf.~\cite[Lemma 1.3]{Wah85}.
Our proof of Proposition~\ref{prp:ext along exc-reduction} follows the lines of \cite[Section 7.D]{GKK10}.

\subsection{The first Chern class} 

We collect some well-known facts about the first Chern class of a line bundle.

\begin{ntn}[First Chern class]
  Let $X$ be a smooth variety and $\sL \in \Pic(X)$ a line bundle. The first Chern
  class $c_1(\sL) \in H^1(X, \Omega_X^1)$ is the image of $\sL$ under the map $\Pic(X) = H^1(X, \O_X^*) \to H^1(X, \Omega_X^1)$ induced by the map $\dif\log\!: \O_X^* \to \Omega_X^1$ that sends $f \mapsto f^{-1} \dif f$.
\end{ntn}

\begin{lem}[Connecting homomorphism of the residue sequence] \label{lem:Connecting homomorphism}
  Let $X$ be a smooth variety and $E \subset X$ an snc divisor, consisting of
  irreducible components $E_1, \dots, E_k$. Consider the short exact sequence
  \begin{sequation}\label{eqn:386}
    0 \to \Omega_X^1 \to \Omega_X^1(\log E) \to \bigoplus_{i=1}^k \O_{E_i} \to 0
  \end{sequation}%
  given by the residue map (cf.\ \cite[2.3(a)]{EV92}). The associated connecting
  homomorphism
  \[ 
  \delta\!: \bigoplus_{i=1}^k H^0(E_i, \O_{E_i}) \to H^1(X, \Omega_X^1) 
  \] 
  sends
  \[ 
  \mathbf 1_{E_i} \mapsto c_1(\O_X(E_i)), \quad 1 \le i \le k. 
  \] 
  Here $\mathbf 1_{E_i}$ denotes the function that is constant with value $1$ on
  $E_i$ and vanishes on the other components.
\end{lem}

\begin{proof}
This is well-known, and easy to prove by a \v Cech cohomology computation.
\end{proof}

For a proof of the following fact see \cite[Paragraph 17]{For81}.

\begin{fct}[Residue map on curves] \label{fct:Residue map}
  Let $C$ be a smooth projective curve. Then there is a canonically defined linear
  map $\Res\!: H^1(C, \Omega_C^1) \to \C$, which is an isomorphism. \qed
\end{fct}

\begin{lem}[Residue and degree] \label{lem:Res circ c_1 = deg}
  If $\sL \in \Pic(C)$ is a line bundle on a smooth projective curve, then
  $\Res(c_1(\ms L)) = \deg \ms L$.
\end{lem} 

\begin{proof} 
  For $P \in C$ a point and $\ms L = \O_C(P)$, the claim is easily seen to be true
  from the description of $\Res$ given in \cite[Thm.~17.3]{For81}. By linearity, this
  is enough.
\end{proof}

\subsection{Proof of Proposition~\ref{prp:ext along exc-pt}} 

We may assume $X$ to be affine of dimension $\ge 2$. Let $E_1, \dots, E_k$ be the
irreducible components of $E$. Consider the short exact sequence \eqref{eqn:386},
\[ 0 \to \Omega_{\wt X}^1 \to \Omega_{\wt X}^1(\log E) \to \bigoplus_{i=1}^k \O_{E_i}
\to 0. \] By the corresponding long exact sequence, it suffices to show injectivity
of the induced map
\[ \delta\!: \bigoplus_{i=1}^k H^0(E_i, \O_{E_i}) \to H^1(\wt X, \Omega_{\wt X}^1).
\]

Note that we may assume $\pi$ to be a projective morphism and then $\wt X$ is quasi-projective.
After choosing a (locally closed) embedding of $\wt X$ into a projective space,
let $H \subset \wt X$ be the intersection of general hyperplanes $H_1, \dots, H_{\dim
  X - 2} \subset \wt X$. (If $X$ is a surface, then $H = \wt X$.) We formulate the
properties of $H$ in a separate lemma.

\begin{lem}\label{lem:Properties of H}
  Using the notation introduced above we have that $(H, E|_H)$ is an snc surface
  pair. Furthermore, for any $i$, $C_i := E_i|_H$ is irreducible (in particular,
  nonempty), and $\pi|_H$ is proper and birational onto its image.
\end{lem}

\begin{proof} 
  If $\dim X=2$, then there is nothing to prove, so we may assume that $\dim X\geq
  3$. In particular the intersection of two irreducible divisors on $\wt X$ is still positive dimensional and hence if one of them is ample, then the intersection is connected. We will use this fact below.

  We proceed inductively, cutting by one hyperplane at a time. First we cut by $H_1$.
  By Bertini's theorem $E + H_1$ is snc and $H_1$ is connected and hence irreducible
  by the above discussion. Hence $(H_1, E|_{H_1})$ is an snc pair, and the
  $E_i|_{H_1}$ are smooth and irreducible for all $i$ by Bertini again. It is clear
  that $\pi|_{H_1}$ is proper and since $H_1$ is general, $\pi|_{H_1}$ is birational.

  Now we are in the same situation as before cutting by $H_1$, so we may apply the
  same argument again and obtain the statement for $H_1 \cap H_2$. After finitely
  many steps, we arrive at $H$.
\end{proof} 

The image $\pi(H)$ need not be normal, but we may normalize it and get a birational
morphism $H \to \pi(H)^\nu$, which contracts all the $C_i$. Then negative
definiteness (\cite[Lemma 3.40]{KM98}) asserts that the intersection matrix $A :=
(C_i \cdot C_j)$ is invertible.

Recall that we need to show the injectivity of $\delta$. To this end, think of the
$C_i$ as smooth projective curves in $\wt X$, consider the restriction morphism
\[ r\!: H^1(\wt X, \Omega_{\wt X}^1) \to \bigoplus_{i=1}^k H^1(C_i, \Omega_{C_i}^1), \]
and observe that the composition
\[ r \circ \delta\!: \bigoplus\limits_{i=1}^k H^0(E_i, \O_{E_i})
  \to \bigoplus_{i=1}^k H^1(C_i, \Omega_{C_i}^1) \]
is an isomorphism:
On the left-hand side, choose the basis consisting of the functions $\mathbf 1_{E_i}$, and on each summand of $\bigoplus\limits_{i=1}^k H^1(C_i, \Omega_{C_i}^1)$, choose the basis canonically determined by the residue map of Fact~\ref{fct:Residue map}. By Lemmas \ref{lem:Connecting homomorphism} and \ref{lem:Res circ c_1 = deg} the map $r \circ \delta$ with respect to these bases is given by the matrix $A$. We have already noted that this matrix is invertible. \qed

\subsection{Proof of Proposition~\ref{prp:ext along exc-reduction}}

We will make essential use of the following proposition.

\begin{prp}[Negativity lemma, see {\cite[Proposition 7.5]{GKK10}}]\label{prp:GKK10,
    Proposition 7.5} 
  Let $\phi\!: \wt Y \to Y$ be a projective birational morphism between normal
  quasi-projective varieties of dimension $\ge 2$, where $\wt Y$ is smooth. Let $y
  \in Y$ be a point whose preimage $\phi^{-1}(y)$ has (not necessarily pure)
  codimension one and let $F_0, \dots, F_k \subset \phi^{-1}(y)$ be the reduced
  divisorial components. If all the $F_i$ are smooth and $\sum e_i F_i$ is a nonzero
  effective divisor, then there is a $0 \le j \le k$ such that $e_j \ne 0$ and
  $h^0(F_j, \ms O_{\wt Y}(\sum e_i F_i)|_{F_j}) = 0$. \qed
\end{prp} 

\begin{proof}[Proof of Proposition~\ref{prp:ext along exc-reduction}] 
  We may assume $X$ to be affine of dimension $\ge 2$. Let
  \begin{sequation}\label{eqn:521}
    \sigma \in H^0 \big( \wt X, \Omega_{\wt X}^1(\log E) \big)
  \end{sequation}%
  be a logarithmic $1$-form. Assuming Proposition~\ref{prp:ext along exc-pt}, we will
  show that
  \begin{sequation}\label{eqn:525}
    \sigma \in H^0 \big( \wt X, \Omega_{\wt X}^1 \big).
  \end{sequation}%
  To this end, we will consider an irreducible component of $E' \subset E$ for which
  $\dim\pi(E')$ is maximal among the irreducible components of $E$, and for any such
  $E'$ we will show that $\sigma \in H^0 \big( \wt X, \Omega_{\wt X}^1(\log(E - E'))
  \big)$. Then replace $E$ by $E-E'$ and repeat the argument until $E$ disappears.

  We proceed by induction on pairs of numbers $\big( \!\dim X, \,\codim_X \pi(E') \big)$
  ordered lexicographically as indicated in the following table:

  \begin{table}[h]
    \centering
    \begin{tabular}{l|ccccccccccc} 
      No.                & 1 & 2 & 3 & 4 & 5 & 6 & 7 & 8 & 9 & 10 & $\cdots$ \\
      \hline 
      $\dim X$           & 2 & 3 & 3 & 4 & 4 & 4 & 5 & 5 & 5 &  5 & $\cdots$ \\ 
      $\codim_X \pi(E')$ & 2 & 2 & 3 & 2 & 3 & 4 & 2 & 3 & 4 &  5 & $\cdots$ \\ 
    \end{tabular} 
\end{table} 

In order to simplify notation, we number the irreducible components $E_i$ of $E$ such
that $E' = E_0$ and $\pi(E_i) = \pi(E_0)$ if and only if $0 \le i \le k$, for some
$k$. Let $e_i$ be the pole orders of $\sigma$ along the $E_i$. These are the minimal
non-negative numbers such that
\[ 
\begin{array}{c}
  \sigma \in H^0 \big( \wt X, \Omega_{\wt X}^1 \tensor \O_{\wt X}(\sum e_i E_i) \big).
\end{array} 
\]
By \eqref{eqn:521}, we already know all the $e_i$ are either 0 or 1, and our aim is
to show that $e_0 = 0$.

\emph{Start of induction.} This is the case $\dim X = \codim_X \pi(E_0) = 2$. For
surfaces, every exceptional divisor is contracted to a point, so Proposition
\ref{prp:ext along exc-pt} applies.

\emph{Inductive step.} We distinguish two possibilities: the divisor $E_0$ may be
mapped to a point by $\pi$, or it may be mapped to a positive-dimensional variety.

If $\dim \pi(E_0) = 0$, then by the choice of $E_0$, every exceptional divisor
contained in $E$ is contracted to a point, so Proposition~\ref{prp:ext along exc-pt}
applies again.

If $\dim \pi(E_0) > 0$, choose general hyperplanes $H_1, \dots, H_{\dim \pi(E_0)}
\subset X$, let $H$ be the intersection $H_1 \cap \cdots \cap H_{\dim \pi(E_0)}$ and
$\wt H$ the preimage $\pi^{-1}(H)$. Applying \cite[Lemmas 2.23 and 2.24]{GKKP11},
we obtain that $H$ is normal
and $\pi|_{\wt H}$ is a log resolution. The intersection $H \cap \pi(E_0)$ is finite, but nonempty. Shrinking $X$, we may assume that $H \cap \pi(E_0)$ consists of a
single point, say $x$. Now set $F_x = (\pi|_{\wt H})^{-1}(x)$ and
\[ F_{x,i} = F_x \cap E_i = (\pi|_{E_i})^{-1}(x). \] Then $F_x$ is the union of the
$F_{x,i}$.

\begin{clm} \label{clm:761}
The subsets $F_{x,0}, \dots, F_{x,k}$ are smooth, irreducible, and have codimension one in $\wt H$, while the other $F_{x,i}$ are empty.
\end{clm}

\begin{proof}
If $0 \le i \le k$, then being a general fibre of $\pi|_{E_i}$, $F_{x,i}$ is
smooth of dimension $\dim E_i - \dim \pi(E_0) = \dim \wt H - 1$. Since $F_{x,i} = \wt
H \cap E_i$, it is an ample divisor on $E_i$, hence connected and by being smooth it is also irreducible.
On the other hand, if $i > k$, then
by the choice of $E'=E_0$, we have that $\pi(E_0) \not\subset \pi(E_i)$,
and hence $x \not\in \pi(E_i)$ and so $F_{x,i} = \emptyset$.
\end{proof}

Claim~\ref{clm:761} implies that it is possible to apply Proposition~\ref{prp:GKK10, Proposition 7.5} to $\pi|_{\wt H}\!: \wt H \to H$, $x \in H$, and $F_{x,0}, \dots, F_{x,k}$, which we will do later.

Now consider the dual of the normal bundle sequence for $\wt H \subset \wt X$,
\[ \xymatrix{ 0 \ar[r] & N_{\wt H/\wt X}^* \ar[r] & \Omega_{\wt X}^1|_{\wt H}
  \ar[r]^\rho & \Omega_{\wt H}^1 \ar[r] & 0, } \] twist it with $\ms F := \O_{\wt
  H}(\sum e_i E_i|_{\wt H})$, and restrict to $F_{x,j}$, for $0 \le j \le k$:
\[ \xymatrix{ N_{\wt H/\wt X}^* \tensor \ms F \ar[r]^-\alpha \ar[d] & \Omega_{\wt
    X}^1|_{\wt H}
  \tensor \ms F \ar[r]^-\beta \ar[d]^{r_j} & \Omega_{\wt H}^1 \tensor \ms F \ar[d] \\
  N_{\wt H/\wt X}^* \tensor \ms F |_{F_{x,j}} \ar[r]^-{\alpha_j} & \Omega_{\wt
    X}^1|_{\wt H} \tensor \ms F \big|_{F_{x,j}} \ar[r]^-{\beta_j} & \Omega_{\wt H}^1
  \tensor \ms F |_{F_{x,j}}.  } \] Since $H$ has smaller dimension than $X$, the
induction hypothesis gives us that $\beta(\sigma|_{\wt H})$ has no poles, that is
\begin{sequation}\label{eqn:572}
  \beta(\sigma|_{\wt H}) \in H^0(\wt H, \Omega_{\wt H}^1) \subset H^0(\wt H,
  \Omega_{\wt H}^1 \tensor \ms F).
\end{sequation}%
Recall that we want to show that $e_0 = 0$. We will show more generally that $e_j =
0$ for all $0 \le j \le k$. So, assume to the contrary that there is an index $j$
with $e_j = 1$. By the definition of the $e_i$, $\sigma|_{\wt H}$ as a section in
$\Omega_{\wt X}^1|_{\wt H} \tensor \ms F$ does not vanish along $F_{x,j}$. But by
(\ref{eqn:572}), $\beta(\sigma|_{\wt H})$ does vanish along $F_{x,j}$. So
$r_j(\sigma|_{\wt H})$ is a nonzero global section in $\ker \beta_j$, which means
that $H^0(F_{x,j}, N_{\wt H/\wt X}^* \tensor \ms F|_{F_{x,j}}) \ne 0$.

Now note that $N_{\wt H/\wt X}|_{F_{x,j}}$ is trivial, because $N_{\wt H/\wt X}$ is
the pullback of $N_{H/X}$.  Hence from $H^0(F_{x,j}, N_{\wt H/\wt X}^* \tensor \ms
F|_{F_{x,j}}) \ne 0$ it follows that $H^0(F_{x,j}, \ms F|_{F_{x,j}}) \ne 0$. Since
this holds for all $j$ with $e_j = 1$, we have a contradiction to Proposition
\ref{prp:GKK10, Proposition 7.5}, showing in particular that $e_0 = 0$ and thus
completing the proof of Proposition~\ref{prp:ext along exc-reduction}.
\end{proof}

\section{Proof of \protect{Theorem~\ref{thm:LZ ext}}}

The aim of the present section is to prove Theorem~\ref{thm:LZ ext}. First, for the reader's convenience we recall some facts about resolutions of singularities.

\begin{lem}[Reflexivity is independent of the choice of resolution] \label{lem:refl indep}
Let $X$ be a normal variety such that $\pi_* \Omega_{\wt X}^1$ is reflexive for some resolution $\pi\!: \wt X \to X$. Then $\phi_* \Omega_{X'}^1$ is reflexive for any resolution $\phi\!: X' \to X$.
\end{lem}

\begin{proof}
Let $\psi\!: \wh X \to X$ be a resolution of $X$ that dominates both $\wt X$ and $X'$, i.e.~we have the following commutative diagram.
\[ \xymatrix{
& \wh X \ar[dl]_{\wh\pi} \ar[dr]^{\wh\phi} \ar[dd]^\psi & \\
\wt X \ar[dr]_\pi & & X' \ar[dl]^\phi \\
& X &
} \]
Since $\wt X$ and $X'$ are smooth, we have $\wh\pi_* \Omega_{\wh X}^1 = \Omega_{\wt X}^1$ and $\wh\phi_* \Omega_{\wh X}^1 = \Omega_{X'}^1$. We obtain
\[ \phi_* \Omega_{X'}^1 = \phi_* \wh\phi_* \Omega_{\wh X}^1 =
   \psi_* \Omega_{\wh X}^1 = \pi_* \wh\pi_* \Omega_{\wh X}^1 =
   \pi_* \Omega_{\wt X}^1. \]
Since $\pi_* \Omega_{\wt X}^1$ is reflexive by assumption, so is $\phi_* \Omega_{X'}^1$.
\end{proof}

The next theorem is a special case of~\cite[Cor.~4.7]{GKK10}.

\begin{thm}[Functorial resolutions] \label{thm:funct res}
Let $X$ be a normal variety. Then there exists a resolution $\phi\!: X' \to X$ with the property that $\phi_* \sT_{X'}$ is reflexive, i.e.~for any vector field $\xi$ on some open subset $U \subset X$, there is a vector field $\wt\xi$ on $\phi^{-1}(U)$ that agrees with $\xi$ wherever $\phi$ is an isomorphism.
\end{thm}

\begin{proof}[Sketch of proof]
It is a classical fact~\cite[Satz 3]{Kau65} that vector fields on $X$ are in one-to-one correspondence with \emph{local $\C$-actions} on $X$. Loosely speaking, a local $\C$-action is a $\C$-action such that $t \bullet z$ is only defined for sufficiently small values of $|t|$, dependent on $z$. For any local $\C$-action, the action map $\C \x X \supset U \to X$ is a smooth morphism.

By~\cite[Thm.~3.45]{Kol07}, there exists a \emph{resolution functor} $\sR$ which to any variety assigns a resolution in such a way that smooth morphisms between varieties can be lifted to the resolutions. The resolutions output by $\sR$ are called \emph{functorial resolutions}. Let $\phi\!: X' \to X$ be the functorial resolution of $X$. Applying the functor $\sR$ to the action map associated to a vector field $\xi$ on $X$, we obtain a diagram
\[ \xymatrix{
\C \x X' \ar[d]_{\id \x \phi} & U' \ar@{ il->}[l] \ar[d] \ar[r] & X' \ar[d]^\phi \\
\C \x X                     & U  \ar@{ il->}[l] \ar[r]        & X.
} \]
One then checks that the map $U' \to X'$ is a local $\C$-action, giving rise to a vector field $\wt\xi$ on $X'$ which extends $\xi$ as desired.

For a rigorous proof of Theorem~\ref{thm:funct res}, the reader may consult~\cite[Sec.~4]{GKK10}.
\end{proof}

\begin{proof}[Proof of Theorem~\ref{thm:LZ ext}]
By assumption, we have a log resolution $\pi\!: \wt X \to X$ such that
$\pi_* \Omega_{\wt X}^1(\log \wt D)$ is reflexive.
We may uniquely write $\wt D = \wt D_{\mathit{bir}} + E$, where $E$ is exceptional and no component of $\wt D_{\mathit{bir}}$ is exceptional.
By Theorem~\ref{thm:drop-a-non-exceptional-div}, $\pi_* \Omega_{\wt X}^1(\log E)$ is reflexive.
Then by Theorem~\ref{thm:ext along exc}, also $\pi_* \Omega_{\wt X}^1$ is reflexive.

Let $\phi\!: X' \to X$ be the functorial resolution from Theorem~\ref{thm:funct res},
so that vector fields on $X$ can be lifted to $X'$.
By Lemma~\ref{lem:refl indep}, $\phi_* \Omega_{X'}^1$ is reflexive.
Now the proof of \cite[Theorem 6.1]{GKKP11} applies verbatim to show that if $\sT_X$ is locally free, then $X$ is smooth.
\end{proof}

\section{Proof of Theorem~\ref{thm:Ext of 1-forms}}

By Theorem~\ref{thm:drop-a-non-exceptional-div} we may assume that
$\wt D = \pi^{-1}(\rd D.)_{\red}$. Let $E$ denote the exceptional locus of $\pi$.
Let $\wh D = \pi^{-1}_* \rd D. + E$. Note that $\wh D$ is obtained from $\wt D$ by adding finitely many irreducible $\pi$-exceptional divisors whose image via $\pi$ is not contained in $\rd D.$. By \cite[Thm.~16.1]{GKKP11}, we know that
$\pi_* \Omega_{\wt X}^1(\log \wh D)$
is reflexive. 
Let $E_1$ be an irreducible component of $\wh D$ that is not contained in $\wt D$.
As we observed above, this means that $\pi(E_1) \not\subset \rd D.$, so by localizing near the general point of $\pi(E_1)$, that is, by further shrinking $X$, we may assume that $\rd D. = \emptyset$.
In this case, $\wh D$ is $\pi$-exceptional, hence Theorem~\ref{thm:ext along exc} implies that $\pi_* \Omega_{\wt X}^1(\log \wh D - E_1)$ is reflexive.
We may iterate this process as long as $\wh D$ is larger than $\wt D$ and so the statement follows. \qed

\section{Optimality of Theorem~\ref{thm:Ext of 1-forms}}\label{sec:Optimality}

In this section, we show that extension of differential forms as in Theorem \ref{thm:Ext of 1-forms} fails in many cases if one shrinks the pole divisor $\wt D$ further, or if one considers forms of higher degree.

\subsection{Shrinking $\wt D$ further}

The aim of this subsection is to prove Theorem~\ref{thm:opt}.
First we need two lemmas.

\begin{lem}[Strictly logarithmic poles] \label{lem:624}
  Let $X$ be a smooth variety and $f \in \O_X(X)$ a regular function such that its
  reduced zero set, $D = \{ f = 0 \}_\red \subset X$, is a divisor with simple normal
  crossings. Then $\dif\log f \in H^0(X, \Omega_X^{1}(\log D))$, and $\dif\log f
  \not\in H^0(X, \Omega_X^{1}(\log B))$ for any reduced divisor $0 \le B < D$.
\end{lem}

\begin{proof}
  This follows directly from the definition of logarithmic differentials. 
\end{proof}

\begin{lem}[Non-reflexivity] \label{lem:628}
  Let $(X, \Sigma)$ be a pair, where $\Sigma$ is a reduced divisor, $\pi\!: \wt X \to
  X$ a log resolution of $(X, \Sigma)$, and $\wt D$ the largest reduced divisor
  contained in $\pi^{-1}(\Sigma)$. Let $E_0$ be an irreducible $\pi$-exceptional
  divisor that is mapped into an effective
  divisor $D$ whose support is contained in $\Sigma$ and which is \Q-Cartier at the
  general point of $\pi(E)$. Then the sheaf $\pi_* \Omega_{\wt X}^1(\log \bb)$ is
  not reflexive for any reduced divisor $\pi^{-1}_* \Sigma \le \bb \le
  \wt D - E_0$. %
\end{lem}

\begin{proof}
  By the assumptions, there is an open set $U \subset X$ with $\pi(E_0) \cap U \ne
  \emptyset$, and a function $f \in \O_X(U)$ cutting out some multiple of $D$. Set $g
  = \pi^* f \in \O_{\wt X}(\pi^{-1}(U))$. By Lemma~\ref{lem:624},
  \[ 
  \dif\log g \in H^0 \big( U \minus \pi(\Exc(\pi)), \pi_* \Omega_{\wt X}^1(\log \bb)
  \big) 
  \] 
  but
  \[ 
  \dif\log g \not\in H^0 \big( U, \pi_* \Omega_{\wt X}^1(\log \bb) \big). 
  \]
  So $\dif\log g$ cannot be extended over the codimension $\ge 2$ subset
  $\pi(\Exc(\pi))$. This implies that $\pi_* \Omega_{\wt X}^1(\log \bb)$ is not
  reflexive.
\end{proof}

\begin{proof}[Proof of Theorem~\ref{thm:opt}]
  Let $E_0$ be a component of $\wt D - \bb$. Then $E_0$ is $\pi$-exceptional and
  $\pi^{-1}_* D \le \bb \le \wt D - E_0$. If we are in case
  (\ref{thm:opt}.\ref{itm:opt-Qfact}), Lemma~\ref{lem:628} applies immediately. Hence
  we may assume we are in case (\ref{thm:opt}.\ref{itm:opt-surface}). Then $\pi(E_0)$
  is a point $p \in X$, and we have $p \in D \subset X$, because $E_0 \subset \wt D
  \subset \pi^{-1}(\rd D.)$ set-theoretically. We will show that $D$ is \Q-Cartier at $p$,
  then the claim follows from Lemma~\ref{lem:628}.

  By shrinking $X$ we may assume that $(X, D)$ is snc away from $p$. For $0 < \eps
  \le 1$, the pair $(X, (1-\eps)D)$ is numerically dlt \cite[Ntn.~4.1,
  Lem.~3.41]{KM98}. By \cite[Prop.~4.11]{KM98}, $X$ is \Q-factorial. In particular,
  $D$ is \Q-Cartier.
\end{proof}

\subsection{Other values of $p$}

The analogue of Theorem~\ref{thm:Ext of 1-forms} does not hold for $p$-forms with $p \ge 2$. Counterexamples may be obtained by taking a $p$-dimensional normal Gorenstein singularity $z \in Z$ which is log canonical but not klt (notice that this only exists if $p \ge 2$), and considering the product $X = Z \x \C^{n-p}$, for $n \ge p$ arbitrary.

Let $\sigma$ be a local generator for $\omega_Z$ and replace $Z$ with a neighbourhood of $z$ where $\sigma$ is everywhere defined.  Then $\mathrm{pr}_1^* \sigma\in
H^0(X_{\reg}, \Omega_{X_{\reg}}^p)$ will not be extendable without logarithmic poles on any resolution of singularities of $X$.

This way one obtains counterexamples to the analogue of Theorem~\ref{thm:Ext of
  1-forms} for any $p\geq 2$ in arbitrary dimension $n\geq p$.

\section{Non-Extension without poles over klt places} \label{sec:klt non-ext}

In this section, we consider a reduced log canonical pair $(X, D)$ and a log
resolution $\pi\!: \wt X \to X$ of $(X, D)$. Deviating slightly from our previous
notation, we let
\begin{equation*}
\begin{array}{rcl}
  E & = & \pi^{-1}_* D + \Exc(\pi), \\
  E_\nklt & = & \text{sum of all divisors in $E$ with discrepancy $-1$}, \\
  E_\nklt \vee \pi^{-1}(D) & $=$ & \text{sum of all divisors contained in $E_\nklt$
    or in $\pi^{-1}(D)$}, \\
  \wt D & = & \text{largest reduced divisor contained in $\pi^{-1}(\text{non-klt
      locus})$}. 
\end{array}
\end{equation*}
Then we obviously have
\[ E_\nklt \;\;\le\;\; E_\nklt \vee \pi^{-1}(D) \;\;\le\;\; \wt D \;\;\le\;\; E. \]

The extension theorem \cite[Thm.~1.5]{GKKP11} states that the sheaves $\pi_*
\Omega_{\wt X}^p (\log \wt D)$ are reflexive for all values of $p$. In
\cite[Section~3.B]{GKKP11}, it was observed that basically by the definition of
discrepancy, even the sheaf $\pi_* \Omega_{\wt X}^n (\log E_\nklt)$ is reflexive,
where $n = \dim X$. This leads to the following natural question.

\begin{ques}\label{ques:opt-ext}
  Are the sheaves $\pi_* \Omega_{\wt X}^p (\log E_\nklt)$ also reflexive when $p <
  n$?
\end{ques}

The answer turns out to be ``no''. However, in the counterexample given in
\cite[Ex.~3.2]{GKKP11}, $X$ is the quadric cone, $D$ consists of two rulings, and the
exceptional divisor where extension fails is contained in the preimage of $D$. This
means that \cite[Ex.~3.2]{GKKP11} does not answer the following refined version of
Question~\ref{ques:opt-ext}:

\begin{ques}\label{ques:opt-ext-refined} 
  Are the sheaves $\pi_* \Omega_{\wt X}^p \bigl( \log \bigl( E_\nklt \vee \pi^{-1}(D)
  \bigr) \bigr)$ reflexive when $p < n$?
\end{ques}

In this section, we will give an example showing that even this question has to be
answered negatively. First we need a lemma.

\begin{lem}[Cusp singularities]\label{lem:cusp sg} 
  There exists a log canonical Gorenstein surface singularity $0 \in S$ that has a
  log resolution $\wt S \to S$ containing two distinct exceptional curves $C_1, C_2$
  which both have discrepancy $-1$ and whose intersection is non-empty.
\end{lem}

\begin{proof}
  This follows from the classification of log canonical Gorenstein surface
  singularities \cite[Sec.~9]{Kaw88}. It is also possible to explicitly write down a
  hypersurface singularity with the desired property. Namely, consider the following
  polynomial in three variables:
  \[ f(x, y, z) = x^2(x+z) - y^2z + z^4. \] A tedious but routine calculation shows
  that the origin in $\C^3$ is an isolated singular point of $S = \{ f = 0 \} \subset
  \C^3$, so $0 \in S$ is a normal Gorenstein surface singularity. Furthermore,
  blowing up $0 \in X$ yields a resolution whose exceptional locus consists of a
  rational curve $C_1$ with a single node, and which has discrepancy $-1$. Blowing up
  that node, we obtain a log resolution containing an additional exceptional rational
  curve $C_2$, also of discrepancy $-1$, such that $C_2$ and the strict transform of
  $C_1$ meet in two points.
\end{proof}

The next theorem tells us that the answer to Question~\ref{ques:opt-ext-refined} is ``no''.

\begin{thm}[Non-extension over klt places, cf.~Theorem~\ref{thm:klt
    non-ext}]\label{thm:klt non-ext 2} 
  There exists a three-dimensional reduced log canonical pair $(X, D)$ such that
  using the notation introduced at the beginning of this section, the sheaf $\pi_*
  \Omega_{\wt X}^2 \bigl( \log \bigl( E_\nklt \vee \pi^{-1}(D) \bigr) \bigr)$ is not
  reflexive.
\end{thm}

\begin{proof}
  Let $0 \in S$ and $p \in C_1 \cap C_2 \subset \wt S$ be as in Lemma~\ref{lem:cusp
    sg}.  Take $X := S \x \A_{\C}^1$ and $D = \emptyset$. Then $X' := \wt S \x
  \A_{\C}^1 \to X$ is a log resolution of $(X, D)$. On $X'$, blow up a point of the
  form $(p, t)$ with $t \in \A_{\C}^1$ arbitrary to obtain $f\!: \wt X \to X'$.  This
  gives a log resolution $\pi\!: \wt X \to X$ of $(X, D)$. Denote by $E_0 \subset \wt
  X$ the exceptional divisor arising from blowing up the point $(p, t)$, and note
  that its discrepancy $a(E_0, X, D) = 0$ by \cite[Lemmata 2.29 and 2.30]{KM98}.
  Hence we have the following diagram.
  \[ 
  \xymatrix{%
    X = S \x \A_{\C}^1 \ar[d] & X' = \wt S \x \A_{\C}^1 \ar[l] & \wt X \supset E_0
    \ar[l]_-f
    \ar@/^1pc/[ll]^{\pi} \\
    S 
  } 
  \]

  In order to prove the claim, consider a local generator $\sigma$ of the canonical
  sheaf $\omega_S$ in a neighborhood $U$ of $0 \in S$, and denote its pullback to
  $X'$ and $\wt X$ by $\sigma'$ and $\wt \sigma$, respectively. Choose local
  coordinates $(u, v, w)$ on $X'$ and $(x, y, z)$ on $\wt X$ such that $f\!: \wt X
  \to X'$ is given by
  \[ 
  f(x, y, z) = (x, xy, xz) 
  \] 
  in these coordinates. Then $E_0 = \{ x = 0 \}$.  Furthermore, we may assume that
  the exceptional divisor of $X' \to X$ is given by the equation $vw = 0$. Then, up
  to a unit, we have
  \[ 
  \sigma' = \dif\log v \wedge \dif\log w 
  \]
  by construction, so
  \[ 
  \wt \sigma = f^* \sigma' = (\dif\log x + \dif\log y) \wedge (\dif\log x + \dif\log
  z). 
  \] 
  This shows that
  \[ 
  \wt \sigma \in H^0 \bigl( \pi^{-1}(U \x \A_{\C}^1), \, \Omega_{\wt X}^2 ( \log E )
  \bigr)
  \] 
  has a pole along $E_0 = \{ x = 0 \}$. However, since $a(E_0, X, D) = 0$ and $D =
  \emptyset$, we see that $E_0$ is not contained in $E_\nklt \vee \pi^{-1}(D)$. An
  argument similar to the proof of Lemma~\ref{lem:628} now yields that $\pi_*
  \Omega_{\wt X}^2 \bigl( \log \bigl( E_\nklt \vee \pi^{-1}(D) \bigr) \bigr)$ is not
  reflexive.
\end{proof}



\providecommand{\bysame}{\leavevmode\hbox to3em{\hrulefill}\thinspace}
\providecommand{\MR}{\relax\ifhmode\unskip\space\fi MR}
\providecommand{\MRhref}[2]{%
  \href{http://www.ams.org/mathscinet-getitem?mr=#1}{#2}
}
\providecommand{\href}[2]{#2}

\end{document}